\documentclass[a4paper, reqno]{amsart}

\usepackage[utf8]{inputenc}
\usepackage{mathtools}
\usepackage{graphicx}
\usepackage{array, longtable, adjustbox}
\usepackage[table]{xcolor}

\newtheorem{lemma}{Lemma}[section]
\newtheorem{theorem}[lemma]{Theorem}

\newcommand{\C}{\mathbb{C}}
\newcommand{\R}{\mathbb{R}}
\newcommand{\Scal}{\mathcal{S}}

\newcommand{\tr}{\mathrm{tr}}
\newcommand{\std}{\mathrm{std}}
\newcommand{\alt}{\mathrm{alt}}
\newcommand{\trv}{\mathrm{trv}}

\DeclareMathOperator{\Hom}{Hom}

\title{Improving the semidefinite programming bound for the kissing
  number by exploiting polynomial symmetry}

\author{Fabrício Caluza Machado} 

\author{Fernando Mário de Oliveira Filho} 

\thanks{The first author was supported by the São Paulo State Research
  Foundation (FAPESP) under grants 2015/05648-4 and 2014/16058-0. The
  second author was partially supported by FAPESP grant 2013/03447-6}

\address{F.C. Machado and F.M.~de Oliveira Filho, Instituto de
  Matemática e Estatística, Rua do Matão 1010, 05508-090 São Paulo/SP,
  Brazil.}
\email{(fabcm1, fmario)@gmail.com}

\subjclass{52C17, 90C22}

\date{September 16, 2016.}

\begin{document}

\begin{abstract}
  The \textit{kissing number of~$\R^n$} is the maximum number of
  pairwise-nonoverlapping unit spheres that can simultaneously touch a
  central unit sphere. Mittelmann and Vallentin (2010), based on the
  semidefinite programming bound of Bachoc and Vallentin (2008),
  computed the best known upper bounds for the kissing number for
  several values of~$n \leq 23$. In this paper, we exploit the
  symmetry present in the semidefinite programming bound to provide
  improved upper bounds for~$n = 9$, \dots,~$23$.
\end{abstract}

\maketitle
\markboth{F.C. Machado and F.M. de Oliveira Filho}{Improving upper
  bounds for the kissing number through symmetry}


\section{Introduction}

The \emph{kissing number problem} asks for the maximum number $\tau_n$
of pairwise-nonoverlapping unit spheres that can simultaneously touch
a central unit sphere in $n$-dimensional Euclidean space.  Its value
is known only for $n = 1$, 2, 3, 4, 8, and~24. The case $n = 3$ is
already difficult; a detailed proof that $\tau_3 = 12$ appeared only
in~1953, given by Schütte and van der Waerden~\cite{SchutteW1953}.

For~$x$, $y \in \R^n$, denote by~$x \cdot y = x_1 y_1 + \cdots + x_n
y_n$ the Euclidean inner product and let~$S^{n-1} = \{\, x \in \R^n :
x \cdot x = 1\, \}$ be the $(n-1)$-dimensional unit sphere. The
\emph{angular distance} between~$x$, $y \in S^{n-1}$ is~$d(x, y) =
\arccos(x \cdot y)$. A \emph{spherical code with minimum angular
  distance~$\theta$} is a set~$C \subseteq S^{n-1}$ such that~$d(x, y)
\geq \theta$ for all distinct~$x$, $y \in C$. Determining the parameter
\[
A(n, \theta) = \max\{\, |C| : \text{$C \subseteq S^{n-1}$ and~$d(x, y)
    \geq \theta$ for all distinct~$x$, $y \in C$}\,\}
\]
is a problem of interest in communication theory (see Conway and
Sloane~\cite{ConwayS1988}, Chapters~1 and~3). The kissing number~$\tau_n$
equals~$A(n, \pi/3)$.

Delsarte, Goethals, and Seidel~\cite{DelsarteGS1977} proposed an upper
bound for~$A(n, \theta)$, known as the linear programming bound, that
was later used by Odlyzko and Sloane~\cite{OdlyzkoS1979}, and
independently Levenshtein~\cite{Levenshtein1979}, to prove
$\tau_8 = 240$ and $\tau_{24} = 196560$. Musin~\cite{Musin2008} used a
stronger version of this bound to show $\tau_{4} = 24$ and Bachoc and
Vallentin~\cite{BachocV2008} strengthened it further via semidefinite
programming.  Mittelmann and Vallentin~\cite{MittelmannV2010} used the
semidefinite programming bound to provide a table with the best upper
bounds for the kissing number for~$n \leq 24$.

The semidefinite programming bound of Bachoc and Vallentin is based on
an infinite-dimensional polynomial optimization problem. To obtain a
finite optimization problem, the maximum degree of the polynomials
involved is restricted. By exploiting the symmetry displayed by the
polynomials in this problem, using techniques such as the ones
described by Gatermann and Parrilo~\cite{GatermannP2004} and Bachoc,
Gijs\-wijt, Schrijver, and Vallentin~\cite{BachocGSV2012}, it is possible to
use polynomials of higher degree, and as a result one obtains improved
upper bounds for the kissing number in dimensions~9 through~23. The
resulting problems are also more stable and can be solved in less time
in comparison to the problems obtained by Mittelmann and
Vallentin. Finally, the numerical results are rigorously verified
using a method similar to the one presented by Dostert, Guzmán,
Oliveira, and Vallentin~\cite{DostertGOV2015}.


\section{The semidefinite programming bound}\label{sec:sdp-bound}

Let us start by recalling the semidefinite programming bound of Bachoc
and Vallentin~\cite{BachocV2008}. Let~$P_k^n(u)$ denote the Jacobi
polynomial of degree~$k$ and parameters~$((n-3) / 2, (n-3) / 2)$,
normalized so that~$P_k^n(1) = 1$ (for background on orthogonal
polynomials, see e.g.~the book by Szegö~\cite{Szego1975}).

Fix~$d > 0$. Let~$Y_k^n$ be the~$(d - k + 1) \times (d - k + 1)$
matrix whose entries are polynomials on the variables~$u$, $v$, $t$
given by
\[ 
(Y_k^n)_{i,j}(u,v,t) = P_i^{n+2k}(u)P_j^{n+2k}(v)Q_k^{n-1}(u,v,t)
\]
for~$0 \leq i, j \leq d - k$, where
\[ 
Q_k^{n-1}(u,v,t) = \big((1-u^2)(1-v^2)\big)^{k/2}P_k^{n-1}\biggl(\frac{t 
-uv}{\sqrt{(1-u^2)(1-v^2)}}\biggr). 
\]

The symmetric group on three elements~$\Scal_3$ acts on a triple~$(u,
v, t)$ by permuting its components. This induces an action
\begin{equation}
\label{eq:s3-action}
\sigma p(u, v, t) = p(\sigma^{-1} (u, v, t))
\end{equation}
on~$\R[u, v, t]$, where~$\sigma \in \Scal_3$. Matrix~$S_k^n$ is
obtained from~$Y_k^n$ by symmetrization with respect to this action:
\[
S_k^n(u, v, t) = \frac{1}{6} \sum_{\sigma \in \Scal_3} \sigma Y_k^n(u,
v, t).
\]

For square matrices~$A$, $B$ of the same dimensions,
write~$\langle A, B \rangle = \tr(B^t A)$. For a
matrix~$A \in \R^{n \times n}$, we write~$A \succeq 0$ to mean
that~$A$ is positive semidefinite.  Fix a dimension~$n \geq 3$ and an
angle~$\theta$ and let~$\Delta$ be the set of all
triples~$(u, v, t) \in \R^3$ that are possible inner products between
three points in a spherical code in~$S^{n-1}$ of minimum angular
distance~$\theta$, that is, $(u, v, t) \in \Delta$ if and only if
there are points~$x$, $y$, $z \in S^{n-1}$ with pairwise minimum
angular distance at least~$\theta$ such that~$u = x \cdot y$,
$v = x \cdot z$, and~$t = y \cdot z$.  The semidefinite programming
bound of Bachoc and Vallentin~\cite{BachocV2008} for~$A(n, \theta)$ is
given by the following optimization problem, where~$J$ is the all-ones
matrix:
\begin{equation}
\label{eq:sdp-bound}
\begin{array}{rl}
\min&\displaystyle1 + \sum_{k=1}^d a_k + b_{11} + \langle J, F_0\rangle\\[5pt]
&\text{(i)\enspace}\displaystyle\sum_{k=1}^d a_k P_k^n(u) + 2 b_{12} + b_{22}\\[5pt]
&\phantom{(i)\enspace}\displaystyle\qquad {} + 3 \sum_{k=0}^d
  \langle S_k^n(u, u, 1), F_k \rangle \leq -1\quad\text{for~$u \in
  [-1, \cos\theta]$},\\
&\text{(ii)\enspace}\displaystyle b_{22} + \sum_{k=0}^d \langle S_k^n(u, v, t), F_k\rangle \leq
  0\quad\text{for~$(u, v, t) \in \Delta$},\\[5pt]
&a_k \geq 0\quad\text{for~$k = 1$, \dots,~$d$},\\[5pt]
&B = \begin{psmallmatrix}
b_{11}&b_{12}\\
b_{21}&b_{22}
\end{psmallmatrix} \succeq 0,\\[5pt]
&\text{$F_k \in \R^{(d - k + 1) \times (d - k + 1)}$ and~$F_k \succeq
  0$ for~$k = 0$, \dots,~$d$}.
\end{array}
\end{equation}

Bachoc and Vallentin showed the following theorem:

\begin{theorem}
  If~$(a_k, B, F_k)$ is a feasible solution
  of~\eqref{eq:sdp-bound},
  then
\[
A(n, \theta) \leq 1 + \sum_{k=1}^d a_k + b_{11} + \langle J,
F_0\rangle.
\]
\end{theorem}

Problem~\eqref{eq:sdp-bound} has infinitely many constraints of
types~(i) and~(ii). These are polynomial constraints: the right-hand
side of~(i) minus the left-hand side is a univariate polynomial
on~$u$, which is required to be nonnegative on the
interval~$[-1, \cos\theta]$; the situation is similar for~(ii), but
then we have a multivariate polynomial on~$u$, $v$, $t$.

Polynomial constraints such as~(i) and~(ii) can be rewritten with
sum-of-squares polynomials and semidefinite programming. Writing a
(univariate or multivariate) polynomial~$p$ as a sum of squares
\[
p = q_1^2 + \cdots + q_s^2
\]
of polynomials~$q_i$ is a sufficient condition for~$p$ to be
nonnegative everywhere. Similarly, let
\[
D = \{\, x \in \R^n : \text{$g_1(x) \geq 0$, \dots,~$g_m(x) \geq
  0$}\,\},
\]
where the~$g_i$ are polynomials, be a basic and closed semialgebraic
set. A sufficient condition for a multivariate polynomial~$p$ to be
nonnegative on~$D$ is for there to exist sum-of-squares
polynomials~$q_0$, $q_1$, \dots,~$q_m$ such that
\begin{equation}
\label{eq:domain-sos}
p = q_0 + q_1 g_1 + \cdots + q_m g_m.
\end{equation}

Sum-of-squares polynomials can be represented by positive semidefinite
matrices. Indeed, say~$p \in \R[x]$, with~$x = (x_1, \ldots, x_n)$, is
a polynomial of degree~$2d$ and let~$B \subseteq \R[x]$ be the set of
all monomials of degree up to~$d$.  Let~$v_B\colon B \to \R[x]$ be such
that~$v_B(r) = r$ for~$r \in B$. We see~$v_B$ as a vector indexed
by~$B$ whose entries are polynomials, so that~$v_B v_B^t$ is a matrix
whose entry~$(r, s)$, for~$r$, $s \in B$, is the
polynomial~$v_B(r) v_B(s) = rs$.  Then~$p$ is a sum of squares if and
only if there is a positive semidefinite
matrix~$Q\colon B \times B \to \R$ such that
\begin{equation}
\label{eq:sos-rep}
p = v_B^t Q v_B = \langle v_B v_B^t, Q \rangle.
\end{equation}
For~$x \in \R^n$, we also write~$v_B(x)$ for the vector obtained
from~$v_B$ by evaluating every entry on~$x$; analogously,
$(v_B v_B^t)(x)$ is the matrix obtained from~$v_B v_B^t$ by evaluating
every entry on~$x$. So, for~$x \in \R^n$,
\[
p(x) = v_B(x)^t Q v_B(x) = \langle (v_B v_B^t)(x), Q\rangle.
\]

Using this relation, we may rewrite constraints~(i) and~(ii)
of~\eqref{eq:sdp-bound}. Let~$g(u) = (u + 1)(\cos\theta - u)$.
Constraint~(i) can be then rewritten as
\begin{multline}
\label{eq:i-rewrite}
\sum_{k=1}^d a_k P_k^n(u) + 2 b_{12} + b_{22} + 3 \sum_{k=0}^d
  \langle S_k^n(u, u, 1), F_k \rangle\\
 + \langle V_0(u), Q_0 \rangle +
  \langle g(u) V_1(u), Q_1\rangle = -1
\end{multline}
with~$Q_0$, $Q_1 \succeq 0$, where~$V_0 = v_{B_0} v_{B_0}^t$
with~$B_0 = \{1, u, u^2, \ldots, u^d\}$ and~$V_1 = v_{B_1} v_{B_1}^t$
with~$B_1 = \{ 1, u, u^2, \ldots, u^{d-1} \}$, so that the maximum
degree of any polynomial appearing on the left-hand side
of~\eqref{eq:i-rewrite} is~$2d$. Notice that two more variable
matrices have been added to our optimization problem, namely~$Q_0$
and~$Q_1$.

To rewrite constraint~(ii), observe that~$\Delta$ is a basic and
closed semialgebraic set. Indeed, we have
\[
\Delta = \{\, (u, v, t) \in \R^3 : \text{$g_i(u, v, t) \geq 0$
  for~$i = 1$, \dots,~4}\,\},
\]
where
\begin{equation}
\label{eq:delta-polys}
\begin{aligned}
g_1(u, v, t) &= g(u),&g_2(u, v, t) &= g(v),\\
g_3(u, v, t) &= g(t),&g_4(u, v, t) &= 1 + 2uvt - u^2 - v^2 - t^2.
\end{aligned}
\end{equation}
Constraint~(ii) can then be similarly rewritten
using~\eqref{eq:domain-sos}, requiring us to add five more variable
matrices to the problem: one for the polynomial~$q_0$, plus one for
each polynomial multiplying one of the~$g_i$ polynomials that
define~$\Delta$. In the next section we will see that, in order to
exploit the symmetry of the polynomials in the~$S_k^n$ matrices, we
need to use different polynomials to represent~$\Delta$; we will
therefore leave the rewriting of constraint~(ii) for later.

Finally, notice that the identity in~\eqref{eq:i-rewrite} is not a
linear constraint on the entries of the variable matrices, but rather
an identity between polynomials. It can however be represented as
several linear constraints, by taking any basis of~$\R[u]_{\leq 2d}$,
the space of univariate polynomials of degree up to~$2d$, expanding
both left and right-hand sides on this basis, and comparing
coefficients. We have to do something similar for constraint~(ii), but
then we need to use a basis of the space~$\R[u, v, t]_{\leq 2d}$;
in~\S\ref{sec:results} we will discuss our choices for such bases. 

Using sum-of-squares polynomials and their relation with semidefinite
programming, we see therefore how to obtain from~\eqref{eq:sdp-bound}
a semidefinite programming problem any feasible solution of which
provides an upper bound for~$A(n, \theta)$.


\section{Exploiting symmetry}

If we rewrite constraint~(ii) of~\eqref{eq:sdp-bound} using
sum-of-squares polynomials as in~\eqref{eq:domain-sos}, then the
largest variable matrix we need will be indexed by all monomials on
variables~$u$, $v$, $t$ of degree at most~$d$. There
are~$\binom{d+3}{3}$ such monomials, hence for~$d = 15$ the largest
matrix will be~$816 \times 816$. So even for moderate values of~$d$ we
get quite large problems that cannot be easily solved in practice.

The polynomials occurring in the~$S_k^n$ matrices are however
invariant under the action~\eqref{eq:s3-action} of~$\Scal_3$. Thanks
to this fact it is possible to block-diagonalize the matrices needed
to represent sum-of-squares polynomials when rewriting
constraint~(ii), and this leads us to smaller and more stable
problems: the block structure of a variable matrix can be informed to
the solver and is used to speed up computations. (The general theory
of symmetry reduction for semidefinite programming has been described
e.g.\ by Bachoc, Gijswijt, Schrijver, and Vallentin~\cite{BachocGSV2012};
Gatermann and Parrilo~\cite{GatermannP2004} deal with the case of
sum-of-squares problems.)

The left-hand side of constraint~(ii) is an invariant polynomial that
should be nonpositive on~$\Delta$. A sufficient condition for this to
hold is for there to exist sum-of-squares polynomials~$q_0$,
\dots,~$q_4$ such that
\begin{equation}
\label{eq:ii-ansatz}
b_{22} + \sum_{k=0}^d \langle S_k^n, F_k\rangle + q_0 + q_1
g_1 + \cdots + q_4 g_4 = 0,
\end{equation}
with~$g_i$ as in~\eqref{eq:delta-polys}. The issue here is that,
though the entries of the~$S_k^n$ matrices are invariant,
polynomials~$g_i$ are not, and hence the~$q_i$ polynomials cannot be
taken to be invariant. The domain~$\Delta$ is itself invariant
however, and we may represent it with invariant polynomials.

\begin{lemma}
\label{lem:inv-delta}
Consider the polynomials
\begin{equation}
\label{eq:inv-delta-polys}
\begin{aligned}
s_1 &= g_1 + g_2 + g_3,&s_2 &= g_1 g_2 + g_1 g_3 + g_2 g_3,\\
s_3 &= g_1 g_2 g_3,&s_4 &= g_4,
\end{aligned}
\end{equation}
with~$g_i$ as in~\eqref{eq:delta-polys}. Then
\[
\Delta = \{\, (u, v, t) \in \R^3 : \text{$s_i(u, v, t) \geq 0$
  for~$i = 1$, \dots,~4}\,\}.
\]
\end{lemma}

\begin{proof}
Since~$s_1$, \dots,~$s_4$ are positive combinations of products
of~$g_1$, \dots,~$g_4$, we have that $g_i(u,v,t) \geq 0$ for
$i = 1$, \dots,~$4$ implies $s_i(u,v,t) \geq 0$ for~$i = 1$,
\dots,~$4$.

For the converse, we may assume that~$g_1(u,v,t) < 0$.
Suppose $s_2(u,v,t)$, $s_3(u, v, t) \geq 0$.
Then $(g_1 g_2 g_3)(u,v,t) \geq 0$ and so $(g_2 g_3)(u,v,t) \leq 0$.
Moreover, $(g_1 g_2 + g_1 g_3 + g_2 g_3)(u,v,t) \geq 0$ implies that
\[
(g_1 (g_2 + g_3))(u,v,t) \geq -(g_2 g_3)(u,v,t) \geq 0,
\]
and so $(g_2 + g_3)(u,v,t) \leq 0$, whence
$s_1(u,v,t) = (g_1 + g_2 + g_3)(u,v,t) < 0$.
\end{proof}

Since the~$s_i$ are invariant, if in~\eqref{eq:ii-ansatz} we replace
the~$g_i$ by~$s_i$, then we may assume without loss of generality that
the~$q_i$ polynomials are also invariant. We may then use the
following theorem in order to represent each polynomial~$q_i$
(cf.~Gatermann and Parrilo~\cite{GatermannP2004}).

\begin{theorem}
\label{thm:block-diag}
For each integer~$d > 0$, there are square matrices~$V^\trv_d$, $V^\alt_d$,
and~$V^\std_d$, whose entries are invariant polynomials in~$\R[u, v,
t]_{\leq 2d}$, such that a polynomial~$p \in \R[u, v, t]_{\leq 2d}$ is
invariant and a sum of squares if and only if there are positive
semidefinite matrices~$Q^\trv$, $Q^\alt$, and~$Q^\std$ of appropriate
sizes satisfying
\[
p = \langle V^\trv_d, Q^\trv\rangle + \langle V^\alt_d, Q^\alt \rangle +
\langle V^\std_d, Q^\std\rangle.
\]
If moreover the dimensions of the matrices~$V^\trv_d$, $V^\alt_d$,
and~$V^\std_d$ are~$a$, $b$, and~$c$, respectively,
then~$\binom{d+3}{3} = a + b + 2c$.
\end{theorem}

Instead of using only one positive semidefinite matrix of
dimension~$\binom{d+3}{3}$, as in~\eqref{eq:sos-rep}, to represent a
sum-of-squares polynomial~$p$ of degree~$2d$, the theorem above
exploits the fact that~$p$ is invariant to represent it with three
smaller matrices of dimensions~$a$, $b$, and~$c$. For~$d = 15$ for
instance we have~$\binom{d+3}{3} = 816$, whereas~$a = 174$, $b = 102$,
and~$c = 270$. These smaller matrices correspond to the
block-diagonalization of the matrix~$Q$ in~\eqref{eq:sos-rep}; each of
them is related to one of the three irreducible representations
of~$\Scal_3$.  A proof of this theorem, together with a description of
how to compute the matrices~$V^\trv_d$, $V^\alt_d$, and~$V^\std_d$,
shall be presented in the next section.

When using the theorem above to rewrite constraint~(ii)
of~\eqref{eq:sdp-bound} we have to choose the degrees of the
polynomials~$q_i$. In this regard, since the left-hand side of~(ii) is
a polynomial of degree at most~$2d$, we choose the degree of~$q_0$ to
be~$2d$ and the degree of~$q_i$, for~$i \geq 1$, to be the largest
possible so that~$s_i q_i$ has degree at most~$2d$. These choices are
important for improving numerical stability and performing the
rigorous verification of results presented in~\S\ref{sec:verify}. The
rewritten constraint is as follows:
\def\foo/#1/#2/#3/{\langle #3 V^\trv_{#1}, #2^\trv\rangle + \langle
  #3 V^\alt_{#1}, #2^\alt\rangle + \langle #3 V^\std_{#1}, #2^\std\rangle}
\begin{equation}\label{eq:ii-rewrite}
\begin{split}
&b_{22} + \sum_{k=0}^d \langle S_k^n, F_k\rangle + \foo/d/R_0//\\
&\qquad+ \foo/d-1/R_1/s_1/\\
&\qquad+ \foo/d-2/R_2/s_2/\\
&\qquad+ \foo/d-3/R_3/s_3/\\
&\qquad+ \foo/d-2/R_4/s_4/ = 0,
\end{split}
\end{equation}
with the~$R$ matrices positive semidefinite.


\section{A proof of Theorem~\ref{thm:block-diag}}
\label{sec:thm-proof}

The proof of Theorem~\ref{thm:block-diag} uses some basic facts from
group representation theory; the reader is referred to the book by
Fulton and Harris~\cite{FultonH2004} for background material.

It is simpler to prove a stronger statement that works for any finite
group~$G$ that acts on~$\R^n$ by permuting coordinates, and for that
we need to work with complex polynomials. Since all irreducible
representations of~$\Scal_3$ are real, however, when~$G = \Scal_3$ we
will be able to use only real polynomials, obtaining
Theorem~\ref{thm:block-diag}.

Say~$G$ is a finite group that acts on~$\R^n$ by permuting
coordinates. This induces for every~$d$ a representation of~$G$
on~$\C[x]_{\leq d}$, where~$x = (x_1, \ldots, x_n)$:
\[
\sigma p(x) = p(\sigma^{-1} x)
\]
for all~$p \in \C[x]_{\leq d}$ and~$\sigma \in G$.

Let~$B$ be the set of all monomials on~$x_1$, \dots,~$x_n$ of degree
at most~$d$. Notice that~$G$ acts on~$B$ by permuting monomials, and
so for each~$\sigma \in G$ there is a permutation
matrix~$P_\sigma\colon B \times B \to \{0,1\}$ such that
\[
v_B(\sigma^{-1} x) = P_\sigma^t v_B(x).
\]

Say~$p = v_B^* Q v_B$ is an invariant polynomial,
where~$Q\colon B \times B \to \C$ is (Hermitian) positive
semidefinite. Then, for~$x \in \R^n$,
\[
\begin{split}
p(x) &= \frac{1}{|G|} \sum_{\sigma \in G} \sigma p(x)\\
&= \frac{1}{|G|} \sum_{\sigma \in G} v_B(\sigma^{-1} x)^* Q
v_B(\sigma^{-1} x)\\
&= \frac{1}{|G|} (P_\sigma^t v_B(x))^* Q (P_\sigma^t v_B(x))\\
&= v_B(x)^* \biggl(\frac{1}{|G|} \sum_{\sigma \in G} P_\sigma Q
P_\sigma^t\biggr) v_B(x).
\end{split}
\]
Now, matrix
\[
\overline{Q} = \frac{1}{|G|} \sum_{\sigma \in G} P_\sigma Q P_\sigma^t
\]
is positive semidefinite and defines a linear transformation
on~$\C[x]_{\leq d}$ that commutes with the action of~$G$:
for~$\sigma \in G$ and~$p \in \C[x]_{\leq d}$ we have
\[
\overline{Q}(\sigma p) = \sigma (\overline{Q} p).
\]

Equip~$\C[x]_{\leq d}$ with the inner product~$(\,\cdot\,,\,\cdot\,)$
for which the standard monomial basis~$B$ is an orthonormal
basis. This inner product is invariant under the action of~$G$, and
the representation of~$G$ on~$\C[x]_{\leq d}$ is unitary with respect
to it. So~$\C[x]_{\leq d}$ decomposes as a direct sum of
pairwise-orthogonal irreducible subspaces
\begin{equation}
\label{eq:cx-dec}
\C[x]_{\leq d} = \bigoplus_{i=1}^r \bigoplus_{k=1}^{h_i} W_{i,k},
\end{equation}
where~$W_{i, k}$ is equivalent to~$W_{j, l}$ if and only if~$i = j$.

The space~$\Hom_G(\C[x]_{\leq d}, \C[x]_{\leq d})$ of linear
transformations on~$\C[x]_{\leq d}$ that commute with the action
of~$G$ can be naturally identified with the space
$(\C[x]^*_{\leq d} \otimes \C[x]_{\leq d})^G$ of tensors that are
invariant under the action of~$G$, and
\begin{equation}
\label{eq:dec}
(\C[x]^*_{\leq d} \otimes \C[x]_{\leq d})^G = \bigoplus_{i,j=1}^r
\bigoplus_{k=1}^{h_i}\bigoplus_{l=1}^{h_j} (W_{i,k}^* \otimes W_{j,l})^G.
\end{equation}
Schur's lemma implies that~$(W_{i,k}^* \otimes W_{j,l})^G$ is~$\{0\}$
when~$i \neq j$, and a one-dimensional space whose elements are
isomorphisms between~$W_{i,k}$ and~$W_{i,l}$ when~$i = j$. For
every~$i = 1$, \dots,~$r$ and~$k = 1$, \dots,~$h_i$, we may choose an
isomorphism $\phi_{i,k} \in (W_{i,1}^* \otimes W_{i,k})^G$ that
preserves the inner product in~$\C[x]_{\leq d}$:
\[
(\phi_{i,k} u, \phi_{i,k} v) = (u, v)\qquad\text{for all~$u$,
  $v \in W_{i,1}$}.
\]
Then~\eqref{eq:dec} simplifies, and
any~$\overline{Q} \in (\C[x]^*_{\leq d} \otimes \C[x]_{\leq d})^G$ can
be written as
\[
\overline{Q} = \sum_{i=1}^r \sum_{k,l=1}^{h_i} \lambda_{i,kl}
\phi_{i,l} \phi_{i,k}^{-1}
\]
for some numbers~$\lambda_{i,kl}$. 

For~$i = 1$, \dots,~$r$, let~$e_{i,1}$, \dots,~$e_{i,n_i}$ be an
orthonormal basis of~$W_{i,1}$. Then for~$k = 1$, \dots,~$h_i$ we have
that~$\phi_{i,k}(e_{i,1})$, \dots,~$\phi_{i,k}(e_{i,n_i})$ is an
orthonormal basis of~$W_{i,k}$. Putting all these bases together, we
get an orthonormal basis of~$\C[x]_{\leq d}$ called \textit{symmetry
  adapted}. Transformation~$\overline{Q}$ has a very special structure
when expressed on this basis: for~$i$, $j = 1$, \dots,~$r$, $k = 1$,
\dots,~$h_i$, $l = 1$, \dots,~$h_j$, $\alpha = 1$, \dots,~$n_i$,
and~$\beta = 1$, \dots,~$n_j$, we have
\begin{equation}
\label{eq:block-structure}
(\overline{Q} \phi_{i,k}(e_{i,\alpha}), \phi_{j,l}(e_{j,\beta}))
= \lambda_{i,kl} \delta_{ij} \delta_{\alpha\beta}.
\end{equation}
In particular, we see that~$\overline{Q}$ is positive semidefinite if
and only if the matrices~$\bigl(\lambda_{i,kl}\bigr)_{k,l=1}^{h_i}$
are positive semidefinite.

For linear transformations~$A$,
$B\colon \C[x]_{\leq d} \to \C[x]_{\leq d}$,
write~$\langle A, B\rangle = \tr(B^* A)$. In view
of~\eqref{eq:block-structure}, for~$x \in \R^n$ we then have
\[
\begin{split}
  p(x) &= \langle (v_B v_B^*)(x), \overline{Q}\rangle\\
  &=\sum_{i,j=1}^r \sum_{k=1}^{h_i} \sum_{l=1}^{h_j}
  \sum_{\alpha=1}^{n_i} \sum_{\beta=1}^{n_j}
  ((v_B v_B^*)(x) \phi_{i,k}(e_{i,\alpha}), \phi_{j,l}(e_{j,\beta}))\\[-1.5em]
  &\phantom{=\sum_{i,j=1}^r \sum_{k=1}^{h_i} \sum_{l=1}^{h_j}
    \sum_{\alpha=1}^{n_i}
    \sum_{\beta=1}^{n_j}}\qquad\qquad\qquad\cdot\overline{(\overline{Q}
    \phi_{i,k}(e_{i,\alpha}),
    \phi_{j,l}(e_{j,\beta}))}\\[-1.5em]
  &=\sum_{i=1}^r \sum_{k,l=1}^{h_i} \overline{\lambda_{i,kl}}
  \sum_{\alpha=1}^{n_i} ((v_B v_B^*)(x) \phi_{i,k}(e_{i,\alpha}),
  \phi_{i,l}(e_{i,\alpha}))\\
  &=\sum_{i=1}^r \sum_{k,l=1}^{h_i} \overline{\lambda_{i,kl}}
  \sum_{\alpha=1}^{n_i} \phi_{i,k}(e_{i,\alpha})(x)
  \overline{\phi_{i,l}(e_{i,\alpha})}(x),
\end{split}
\]
where $\overline{\phi_{i,l}(e_{i,\alpha})}$ is the polynomial
obtained from~$\phi_{i,l}(e_{i,\alpha})$ by conjugating every
coefficient.

So by taking as~$V^i_d$, for~$i = 1$, \dots,~$r$, the matrix whose
entry~$(k, l)$ is equal to the polynomial
\[
\sum_{\alpha=1}^{n_i} \phi_{i,k}(e_{i,\alpha})(x)
  \overline{\phi_{i,l}(e_{i,\alpha})}(x)
\]
we get
\[
p(x) = \sum_{i=1}^r \langle V^i_d(x),
\bigl(\lambda_{i,kl}\bigr)_{k,l=1}^{h_i}\rangle.
\]
Since moreover for any choice of~$\lambda_{i,kl}$ we get, by
construction, an invariant polynomial, the polynomials in the~$V^i_d$
matrices must be invariant. Finally, matrix~$V_d^i$ has
dimension~$h_i$, the multiplicity of~$W_{i,1}$ in the decomposition
of~$\C[x]_{\leq d}$. Hence, if~$N$ is the dimension
of~$\C[x]_{\leq d}$ and~$n_i$ is the dimension of~$W_{i,1}$, then
\[
N = \sum_{i=1}^r n_i h_i.
\]

So we see that each matrix~$V^i_d$ corresponds to one of the
irreducible representations of~$G$ that appear in the decomposition
of~$\C[x]_{\leq d}$. Moreover, all we need to compute~$V^i_d$ is the
symmetry-adapted basis, and for that we need
decomposition~\eqref{eq:cx-dec} and the~$\phi_{i,k}$ isomorphisms,
both of which can be computed using standard linear algebra. In
practice, however, a projection formula such as the one found in~\S2.7
of the book by Serre~\cite{Serre1977} can be used to compute the
symmetry-adapted basis directly, given that we know all irreducible
representations of~$G$.

Matrices~$V^i_d$ might have polynomials with complex coefficients, and
some of the~$\lambda_{i,kl}$ might be complex numbers, even if~$p$ is
a real polynomial. This is unavoidable in general, but when~$G$ has
only real irreducible representations (i.e., representations that can
be expressed by real matrices), all computations involve only real
numbers and the matrices~$V^i_d$ contain only real polynomials; as a
result, all the~$\lambda_{i,kl}$ can be taken real.

Every symmetric group has only real irreducible representations (see
e.g.\ Chapter~4 of the book by Fulton and
Harris~\cite{FultonH2004}). The symmetric group on three
elements,~$\Scal_3$, has only three irreducible representations: the
\textit{trivial} and \textit{alternating} representations, both of
dimension one, and the \textit{standard} representation, of dimension
two. All of them appear in the decomposition~\eqref{eq:cx-dec}
of~$\C[u, v, t]_{\leq d}$, and so we get Theorem~\ref{thm:block-diag}.


\section{Results}
\label{sec:results}

We solve problem~\eqref{eq:sdp-bound} with constraints~(i) and~(ii)
replaced by~\eqref{eq:i-rewrite} and~\eqref{eq:ii-rewrite},
respectively. These constraints are polynomial identities that have to
be expanded on bases of the corresponding vector spaces to produce
linear constraints in the problem variables, as explained
in~\S\ref{sec:sdp-bound}. For constraint~\eqref{eq:i-rewrite}, we
simply take the standard monomial basis of~$\R[u]_{\leq 2d}$. For
constraint~\eqref{eq:ii-rewrite}, we note that all polynomials
involved are invariant, so we have fewer constraints if we use a basis
of the subspace of invariant polynomials of~$\R[u,v,t]_{\leq 2d}$. One
way to find such basis is to consider all triples~$(a,b,c)$ of
nonnegative integers such that~$a +2b + 3c \leq 2d$ and for each
triple take the polynomial~$(u+v+t)^a(u^2+v^2+t^2)^b(u^3+v^3+t^3)^c$.
By Proposition~1.1.2 of Sturmfels~\cite{Sturmfels2008}, these
polynomials generate the subspace of invariant polynomials of degree
at most~$2d$, and by Theorem~1.1.1 of the same book together with a
dimension argument, they actually form a basis of this subspace.

The application of symmetry reduction lead to big improvements in
practice. For instance, the high-precision solver
SDPA-GMP~\cite{Nakata2010} with~200 bits of precision running on
a~2.4GHz processor took~9 days to solve the problem for~$n = 12$
and~$d = 11$ without symmetry reduction. After the reduction, the
resulting semidefinite program could be solved in less than~12 hours.

In this way, it was possible to make computations with~$d$ up to~$16$
within a computing time of~6 weeks and get new upper bounds for the
kissing number on dimensions~9 to~23, improving the results given by
Mittelmann and Vallentin~\cite{MittelmannV2010}. 

\begin{table}[t]
\begin{adjustbox}{center}
{\footnotesize
\begin{tabular}{ccccccccccc}
    &               &     & \textsl{previous}                 &                   & &     &               &     & \textsl{previous}                 &                  \\
$n$ & \textsl{l.b.} & $d$ & \textsl{u.b.~\cite{MittelmannV2010}} & \textsl{new u.b.} & & $n$ & \textsl{l.b.} & $d$ & \textsl{u.b.~\cite{MittelmannV2010}} & \textsl{new u.b.}\\
\rowcolor{gray!20}
3&12&14&12.38180947&12.381921&&14&1606&14&3183.133169&3183.348148\\
\rowcolor{gray!20}
&&15&&12.374682&&&&15&&3180.112464\\
\rowcolor{gray!20}
&&16&&12.368591&&&&16&&\underline{3177}.917052\\
4&24&14&24.06628391&24.066298&&15&2564&14&4866.245659&4866.795537\\
&&15&&24.062758&&&&15&&4862.382161\\
&&16&&24.056903&&&&16&&\underline{4858}.505436\\
\rowcolor{gray!20}
5&40&14&44.99899685&44.999047&&16&4320&14&7355.809036&7356.238006\\
\rowcolor{gray!20}
&&15&&44.987727&&&&15&&7341.324655\\
\rowcolor{gray!20}
&&16&&44.981067&&&&16&&\underline{7332}.776399\\
6&72&14&78.24061272&78.240781&&17&5346&14&11072.37543&11073.844334\\
&&15&&78.212731&&&&15&&11030.170254\\
&&16&&78.187761&&&&16&&\underline{11014}.183845\\
\rowcolor{gray!20}
7&126&14&134.4488169&134.456246&&18&7398&14&16572.26478&16575.934858\\
\rowcolor{gray!20}
&&15&&134.330898&&&&15&&16489.848647\\
\rowcolor{gray!20}
&&16&&134.270201&&&&16&&\underline{16469}.090329\\
9&306&14&364.0919287&364.104934&&19&10668&14&24812.30254&24819.810569\\
&&15&&363.888016&&&&15&&24654.968481\\
&&16&&\underline{363}.675154&&&&16&&\underline{24575}.871259\\
\rowcolor{gray!20}
10&500&14&554.5075418&554.522392&&20&17400&14&36764.40138&36761.630730\\
\rowcolor{gray!20}
&&15&&554.225840&&&&15&&36522.436885\\
\rowcolor{gray!20}
&&16&&\underline{553}.827497&&&&16&&\underline{36402}.675795\\
11&582&14&870.8831157&870.908146&&21&27720&14&54584.76757&54579.036297\\
&&15&&869.874183&&&&15&&54069.067238\\
&&16&&\underline{869}.244985&&&&16&&\underline{53878}.722941\\
\rowcolor{gray!20}
12&840&14&1357.889300&1357.934329&&22&49896&14&82340.08003&82338.035075\\
\rowcolor{gray!20}
&&15&&1357.118955&&&&15&&81688.317095\\
\rowcolor{gray!20}
&&16&&\underline{1356}.603728&&&&16&&\underline{81376}.459564\\
13&1154&14&2069.587585&2069.675634&&23&93150&14&124416.9796&124509.320059\\
&&15&&2067.388613&&&&15&&123756.492951\\
&&16&&\underline{2066}.405173&&&&16&&\underline{123328}.397290\\
\end{tabular}
}
\end{adjustbox}
\bigskip

\caption{Lower and upper bounds (l.b.\ and u.b.) for the kissing number in
  dimensions~3, \dots,~24. Dimensions~8 and~24 are omitted since in
  these dimensions the linear programming bound is tight. All lower bounds
  can be found in the book of Conway and Sloane~\cite{ConwayS1988},
  except for dimensions~13 and~14, in which case they were obtained by
  Ericson and Zinoviev~\cite{ZinovievE1999}. Improvements over previously
  known upper bounds are underlined. All new bounds reported have been
  rigorously verified; see~\S\ref{sec:verify}.}
\label{tab:bounds}
\vskip-5mm{}
\end{table}

The results are shown on Table~\ref{tab:bounds}. Following Mittelmann
and Vallentin, the table includes different values of~$d$ and decimal
digits, since the sequence of values gives a clue about how strong the
bound of Bachoc and Vallentin~\cite{BachocV2008} can be if polynomials
of higher degree are used. This is not the case for the linear
programming bound, where the increase in degree does not give
significant improvements~\cite{OdlyzkoS1979}. Even the decimal digits
in dimension~4 are interesting, since a tight bound can provide
information about the optimal configurations (it is still an open
problem whether the configuration of~24 points in dimension~4 is
unique; for dimensions~8 and~24 uniqueness was proved by Bannai and
Sloane~\cite{BannaiS1981} using the linear programming bound).

Finally, we observe that most values for~$d = 14$ are in fact
bigger than the corresponding values provided by Mittelmann and
Vallentin~\cite{MittelmannV2010}, as the problems solved are not exactly
the same: polynomials~$s_i$ and~$g_i$, used to represent~$\Delta$, are
different.

\section{Rigorous verification of results}
\label{sec:verify}

Floating-point arithmetic is used both in the process of computing the
input to the solver (in particular when computing the symmetry-adapted
basis) and by the solver itself. So the solution obtained by the
solver is likely not feasible and hence its objective value might not
be an upper bound to the kissing number. If the solution is, however,
composed by positive \textit{definite} matrices and is close enough to
being feasible, it is possible to prove that it can be turned into a
feasible solution without changing its objective value, thus showing
that its objective value is an upper bound for the kissing number.

The idea is very similar to the one used by Dostert, Guzmán, Oliveira,
and Vallentin~\cite{DostertGOV2015}. The first step is to find a good
solution to our problem, namely one satisfying the following
condition: \textit{the minimum eigenvalue of any matrix is large
  compared to the maximum violation of any constraint}. (The precise
meaning of ``large'' will be clarified soon.) If this condition is
satisfied, then it is possible to turn the solution into a feasible
one, without changing its objective value.

Next, we need to verify rigorously that the solution satisfies the
condition. It is not enough for such a verification procedure to use
floating-point arithmetic, since then we cannot be sure of the
correctness of the computations. We will see how rigorous bounds on
the minimum eigenvalue of each matrix and also on the violation of
each constraint can be obtained using high-precision interval
arithmetic.

The first step is to obtain a good solution. To get small constraint
violations, we need to use a high-precision solver; we use
SDPA-GMP~\cite{Nakata2010} with~200 bits of precision. Usually,
solvers will return a solution that lies close to the boundary of the
cone of positive semidefinite matrices, and so the minimum eigenvalues
of the solution matrices will be very close to zero. To get a solution
with large minimum eigenvalues, we solve the problem with a change of
variables: we fix~$\lambda_{\min} > 0$ and replace each variable~$X$
by~$X' + \lambda_{\min} I$ with~$X' \succeq 0$. This gives a solution
where~$X$ has minimum eigenvalue at least~$\lambda_{\min}$, but of
course the objective value increases as~$\lambda_{\min}$
increases. Parameter~$\lambda_{\min}$ has to be chosen small enough so
that the loss in objective value is small, but large enough in
comparison to the constraint violations. Choosing an
appropriate~$\lambda_{\min}$ is a matter of trial and error; we
observed that values around~$10^{-8}$ or~$10^{-10}$ work well in
practice. To be able to choose a strictly positive~$\lambda_{\min}$, a
feasible solution consisting of positive definite matrices must
exist. So we need to avoid dependencies in our formulation; this is
one reason why it is important to carefully choose the degrees of the
polynomials appearing in~\eqref{eq:ii-rewrite}.

To carry out the rigorous verification, it is convenient to rewrite
constraint~\eqref{eq:ii-rewrite} without using
Theorem~\ref{thm:block-diag}, that is, using only one large positive
semidefinite matrix for each sum-of-squares polynomial.  If we use the
standard monomial basis~$B_d$ for~$\R[u,v,t]_{\leq d}$, then
matrix~$V_d = v_{B_d}v_{B_d}^t$ is easy to construct and all numbers
appearing in the input are rational. Constraint~\eqref{eq:ii-rewrite}
becomes
\begin{multline}
\label{eq:ii-rerewrite}
 b_{22} + \sum_{k=0}^d \langle S_k^n, F_k\rangle + \langle V_d, R_0\rangle
 + \langle s_1V_{d-1}, R_1\rangle + \langle s_2V_{d-2}, R_2\rangle\\ + 
 \langle s_3V_{d-3}, R_3\rangle + \langle s_4V_{d-2}, R_4\rangle = 0.
\end{multline}

We can convert the solution obtained by the solver for a problem with
constraint~\eqref{eq:ii-rewrite} into a solution for the problem
where~\eqref{eq:ii-rewrite} is replaced
by~\eqref{eq:ii-rerewrite}. Indeed, note that in the process described
in~\S\ref{sec:thm-proof} matrix~$\overline{Q}$ becomes block-diagonal
when expressed in the symmetry-adapted basis
(cf. equation~\eqref{eq:block-structure}), so the conversion between
constraints amounts to a change of basis. The problem size increases,
since the matrices in the sum-of-squares formulation will not be
block-diagonal anymore, but this is not an issue since the problem is
already solved and the conversion is not an expensive operation.

Once we have a good solution to our reformulated problem, it is time
to carry out the verification. For each variable~$X$, we use
high-precision floating-point arithmetic to perform a binary search to
find a large~$\lambda_X > 0$ such that~$X - \lambda_X I$ has a
Cholesky decomposition~$LL^t$. Typically, this~$\lambda_X$ is a bit
smaller than the~$\lambda_{\min}$ used to find the solution. Now, we
convert the floating-point matrix~$L$ to a rational
matrix~$\overline{L}$ and set
\[
\overline{X} = \overline{L}\hskip1pt\overline{L}^t + \lambda_X I,
\]
so that~$\overline{X}$ is a rational matrix. Doing this for every
matrix variable, we obtain a rational almost-feasible solution of our
problem together with a rigorous lower bound on the minimum eigenvalue
of each matrix.

Next we check that the violation of the equality constraints
in~\eqref{eq:i-rewrite} and~\eqref{eq:ii-rerewrite} for our rational
almost-feasible solution is small compared to the minimum
eigenvalues. Both cases are similar, so let us think of
constraint~\eqref{eq:ii-rerewrite}. We now have a rational
polynomial~$r$ that is the left-hand side of~\eqref{eq:ii-rerewrite},
which will likely not be the zero polynomial. Note however that all
monomials of degree at most~$2d$ appear as entries of~$V_d$, so there
is a rational matrix~$A$ such that~$r = \langle V_d, A\rangle$.
Replacing~$\overline{R_0}$ by~$\overline{R_0} - A$, we manage to
satisfy constraint~\eqref{eq:ii-rerewrite}.

To ensure that~$\overline{R_0} - A \succeq 0$ it suffices to require
that~$\|A\| = \langle A, A\rangle^{1/2} \leq \lambda_{R_0}$, and this
condition can be verified directly from~$r$. Notice moreover that
changing~$\overline{R_0}$ does not change the objective value of the
solution. In practice, computing~$\overline{X}$ using rational
arithmetic can be computationally costly. Since we only care about
comparing~$\|A\|$ with the bound on the minimum eigenvalue, we do not
need to use rational arithmetic: it is sufficient to use, say,
high-precision interval arithmetic, as provided for instance by a
library such as MPFI~\cite{RevolR2005}.

The solutions that
  provide all the new upper bounds given on Table~\ref{tab:bounds} as
  well as the verification script described above are available at
\[
\hbox{\tt http://www.ime.usp.br/\~{}fabcm/kissing-number}
\]


\end{document}